\documentclass[10pt,a4paper,reqno]{amsart}
\usepackage[utf8]{inputenc}

\usepackage[makeroom]{cancel}
\usepackage{float}
\usepackage{mathabx}
\usepackage{amsmath}
\usepackage{amsthm}
\usepackage{amssymb}
\usepackage{amscd}
\usepackage{amsfonts}
\usepackage{amsbsy}
\usepackage{mathrsfs}
\usepackage{graphicx}
\usepackage[dvips]{psfrag}
\usepackage{array}
\usepackage{xcolor}
\usepackage{epsfig}
\usepackage{url}
\usepackage{enumitem}
\usepackage{overpic}
\usepackage{epstopdf}
\usepackage{apptools}
\usepackage{amsmath}

\newcommand{\R}{\ensuremath{\mathbb{R}}}

\def\p{\partial}

\newtheorem {theorem} {Theorem}

\newtheorem {lemma}[theorem] {Lemma}

\title[Number of normally hyperbolic limit tori in $3$D polynomial vector field]{Strict increase in the number of normally hyperbolic limit tori in $3$D polynomial vector fields}

\author[ L. Q. Arakaki]{ Lucas Q. Arakaki$^*$}
\address{Departamento de Matem\'{a}tica, Instituto de Matem\'{a}tica, Estatística e Computa\c{c}\~{a}o Cient\'{i}fica (IMECC), Universidade
	Estadual de Campinas (UNICAMP), Rua S\'{e}rgio Buarque de Holanda, 651, Cidade Universit\'{a}ria Zeferino Vaz, 13083--859, Campinas, SP, Brazil.} 
\email{larakaki@unicamp.br}

\author[D. D. Novaes]{Douglas D. Novaes}
\address{Departamento de Matem\'{a}tica, Instituto de Matem\'{a}tica, Estatística e Computa\c{c}\~{a}o Cient\'{i}fica (IMECC), Universidade
	Estadual de Campinas (UNICAMP), Rua S\'{e}rgio Buarque de Holanda, 651, Cidade Universit\'{a}ria Zeferino Vaz, 13083--859, Campinas, SP, Brazil.} \email{ddnovaes@unicamp.br} 

\makeatletter 
\@namedef{subjclassname@2020}{%
	\textup{2020} Mathematics Subject Classification}
\makeatother
\thanks{$^*$ Corresponding author.}



\allowdisplaybreaks

\begin{document}

\begin{abstract}
The second part of Hilbert's 16th problem concerns determining the maximum number $H(m)$ of limit cycles that a planar polynomial vector field of degree $m$ can exhibit. A natural extension to the three-dimensional space is to study the maximum number $N(m)$ of limit tori that can occur in spatial polynomial vector fields of degree $m$. In this work, we focus on normally hyperbolic limit tori and show that the corresponding maximum number $N_h(m)$, if finite, increases strictly with $m$. More precisely, we prove that $N_h(m+1) \geqslant N_h(m) + 1$. Our proof relies on the torus bifurcation phenomenon observed in spatial vector fields near Hopf-Zero equilibria. While conditions for such bifurcations are typically expressed in terms of higher-order normal form coefficients, we derive explicit and verifiable criteria for the occurrence of a torus bifurcation assuming only that the linear part of the unperturbed vector field is in Jordan normal form. This approach circumvents the need for intricate computations involving higher-order normal forms.
\end{abstract}

\maketitle

\section{Introduction and statement of the main results}

The existence of compact invariant manifolds lies at the heart of the qualitative theory of differential systems, as it provides essential insights into their underlying dynamical structure. In planar differential systems, periodic solutions are the first nontrivial examples of compact invariant manifolds and have been widely studied. In higher dimensions, invariant tori play a role similar to that of limit cycles in the plane, allowing for natural extensions of classical questions concerning their existence, number, and stability.

In \cite{NP25}, the authors proposed, as an extension of Hilbert's 16th problem to three-dimensional space, the study of the maximal number $N(m)$ of isolated invariant tori, which we henceforth refer to as \emph{limit tori}, that can occur in polynomial vector fields of degree $m$ in $\mathbb{R}^3$. More precisely, given a vector field $X$ associated to the following polynomial differential system
\begin{equation*}
	\begin{aligned}
		\dot{x} &= P(x,y,z),\\
		\dot{y} &= Q(x,y,z),\\
		\dot{z} &= R(x,y,z),
	\end{aligned}
\end{equation*}
let $\tau(P,Q,R)$ denotes the number of limit tori in its phase space. Then,
\[
N(m) = \sup\{ \tau(P,Q,R) : \deg(P), \deg(Q), \deg(R) \leqslant m \}.
\]
When analyzing invariant compact manifolds in three-dimensional differential systems, the notion of \emph{normal hyperbolicity} plays a fundamental role (see, for instance, \cite{Fenichel1,W94}). One of its key features is that, roughly speaking, normally hyperbolic invariant manifolds persist under small perturbations. In this context, normally hyperbolic limit tori emerge naturally as three-dimensional analogues of hyperbolic limit cycles in planar differential systems. This motivates the definition of
\[
N_h(m) = \sup\{ \tau_h(P,Q,R) : \deg(P), \deg(Q), \deg(R) \leqslant m \},
\]
where $\tau_h(P,Q,R)$ denotes the number of normally hyperbolic limit tori of the vector field $X$.

The bifurcation of limit tori has attracted considerable attention in recent studies. In \cite{CanNov20JDE,Novaes2023,PereiraNovaes23mech}, tools based on \emph{averaging theory} were developed to investigate the existence of such invariant objects. These methods have been employed to study applied models in \cite{CanNov20JDE,CanNovVal20,CandidoVallsVanderpol}. Based on the techniques developed in \cite{PereiraNovaes23mech} for detecting normally hyperbolic limit tori, the authors of \cite{NP25} introduced a mechanism for constructing a three-dimensional vector field from a planar one, in such a way that the number of normally hyperbolic limit tori in the resulting differential system matches the number of hyperbolic limit cycles in the planar differential system. A key feature of this construction is that the three-dimensional vector field remains polynomial whenever the original planar vector field is polynomial. This approach allowed the authors of \cite{NP25} to establish a first connection between $N_h(m)$ and the Hilbert number $H(m)$, namely,
\[
N_h(m) \geqslant H\left(\left\lfloor \frac{m}{2} \right\rfloor - 1\right),
\]
as well as a first estimate for the asymptotic growth of $N_h(m)$, showing that it grows at least as fast as $m^3/128$.

A natural next step in the study of $N(m)$ is to investigate which properties known for the Hilbert number $H(m)$ are also satisfied by $N(m)$. In \cite{SantanaGasull}, the authors proved that, if finite, $H(m)$ is a strictly increasing function. Inspired by their construction, we show here that the same property holds for $N_h(m)$. Our first main result is stated below.

\begin{theorem}\label{Teo:StrictIncrease}
If $N_h(m)<\infty$ for some $m\in\mathbb{N}$, then $N_h(m+1)\geqslant N_h(m)+1$.
\end{theorem}

 Theorem \ref{Teo:StrictIncrease} is proved in Section \ref{Sec:StrictIncrease}. The argument is based on the torus bifurcation phenomenon exhibited by three-dimensional vector fields near Hopf-Zero equilibria, singularities whose eigenvalues are given by $\{0, \omega i, -\omega i\}$. Conditions for the occurrence of such bifurcations are usually expressed in terms of the coefficients of higher-order normal forms of the vector field near the equilibrium point (see, for instance, \cite[Theorem 3]{G81}, and \cite[Corollary 2]{L79}). 

In this work, we introduce a novel criterion that provides explicit conditions for the occurrence of a torus bifurcation near Hopf-Zero equilibria. Notably, we assume only that the linear part of the unperturbed vector field is in Jordan normal form, thereby avoiding the intricate computations involved in deriving higher-order normal forms. This formulation is particularly well suited to our setting, where a more direct criterion is required.

More specifically, we consider differential systems of the form
\begin{equation}\label{unp}
\begin{aligned}
\dot{x} &= -y + P(x,y,z),\\
\dot{y} &= x + Q(x,y,z),\\
\dot{z} &= R(x,y,z),
\end{aligned}
\end{equation}
where $P$, $Q$, and $R$ are $C^3$ functions without constant or linear terms. Under the nondegeneracy condition
\begin{equation}\label{eq:gencond}
\Omega := -\left(\dfrac{\p^2 P}{\p x \p z}(0,0,0) + \dfrac{\p^2 Q}{\p y \p z}(0,0,0)\right)
\left(\dfrac{\p^2 R}{\p 
x^2}(0,0,0) + \dfrac{\p^2 R}{\p y^2}(0,0,0)\right) > 0,
\end{equation}
our second main result establishes explicit conditions for a torus bifurcation to occur in perturbations of the differential system \eqref{unp} within the two-parameter family
\begin{equation}\label{eq:app0Hopfperturbado}
\begin{aligned}
\dot{x} &= -y + P(x,y,z) + \varepsilon U(x,y,z;\mu,\varepsilon), \\
\dot{y} &= x + Q(x,y,z) + \varepsilon V(x,y,z;\mu,\varepsilon), \\
\dot{z} &= R(x,y,z) + \varepsilon W(x,y,z;\mu,\varepsilon),
\end{aligned}
\end{equation}
giving rise to a normally hyperbolic limit torus from the origin. Here, $\mu \in J \subset \mathbb{R}$, where $J$ is an open interval, $\varepsilon \in (-\varepsilon_0, \varepsilon_0)$, and $U, V, W$ are $C^3$ functions.

For the sake of conciseness, we adopt the following notation. Given a $C^r$ function $F: \mathbb{R}^3 \to \mathbb{R}$, we denote by $F^{(j,k,l)}$ the partial derivative $\partial^{j+k+l} F / \partial x^j \partial y^k \partial z^l$ evaluated at $(0,0,0)$. When $F$ also depends on a parameter $\mu$, we write $F^{(j,k,l)}(\mu)$ for the corresponding partial derivative evaluated at $(0,0,0,\mu)$. In what follows, $U_i$, $V_i$, and $W_i$ denote the coefficients of $\varepsilon^i$ in the power series expansions of $U$, $V$, and $W$, respectively.

\begin{theorem}\label{Teo:0Hopfbif}
Let $P,Q,R:\R^3\to\R^3$ be $C^3$ functions with no constant or linear terms for which \eqref{eq:gencond} holds and $U, V, W:
\R^3\times\R^2\to\R^3$ be $C^3$ functions as defined above.
Denote
\begin{equation}\label{eq:appcond1}
\begin{aligned}
\Gamma(\mu):=&\displaystyle \frac{2R^{(0,0,2)}(U_1^{(1,0,0)}(\mu)+V_1^{(0,1,0)}(\mu))^2-W_1^{(0,0,1)}(\mu)
   (U_1^{(1,0,0)}(\mu)+V_1^{(0,1,0)}(\mu))}{(R^{(0,2,0)}+R^{(2,0,0)})
   (P^{(1,0,1)}+Q^{(0,1,1)})^2}\\
   &\displaystyle+\frac{4W_2(0,0,0;\mu)}{(R^{(0,2,0)}+R^{(2,0,0)})},
\end{aligned}
\end{equation}
and
\begin{equation*}\label{eq:appcond2}
\eta(\mu):= \dfrac{\pi\left(
   \left(P^{(1,0,1)}+Q^{(0,1,1)}\right)W_1^{(0,0,1)}(\mu)-  R^{(0,0,2)}\left(U_1^{(1,0,0)}(\mu)+V_1^{(0,1,0)}(\mu)\right)\right)}{ \left(P^{(1,0,1)}+Q^{(0,1,1)}\right)}.
\end{equation*}
Assume that $U_1(0,0,0;\mu) = V_1(0,0,0;\mu) = W_1(0,0,0;\mu) = 0$ and $\Gamma(\mu) < 0$ for every $\mu \in J$. Suppose further that there exists $\mu_0 \in J$ such that $\eta(\mu_0) = 0$ and $\eta'(\mu_0) \neq 0$. Then, there exist a quantity $\ell_1$, depending only on $P$, $Q$, and $R$ --- see \eqref{eq:Nh_Lyapunov} for the explicit expression of $\ell_1$ --- and a smooth curve $\mu(\varepsilon)$ defined in a neighborhood of $0$, with $\mu(0) = \mu_0$, such that for each sufficiently small $\varepsilon > 0$, there exists an interval $I_\varepsilon \subset J$ containing $\mu(\varepsilon)$ with the following property: for every $\mu \in I_{\varepsilon}$ such that $(\mu - \mu(\varepsilon))\ell_1 < 0$, the differential system \eqref{eq:app0Hopfperturbado} has a normally hyperbolic limit torus surrounding a periodic solution. Moreover, this limit torus bifurcates from the periodic solution as $\mu$ crosses $\mu(\varepsilon)$, while the periodic solution itself bifurcates from the origin as $\varepsilon$ crosses $0$.
\end{theorem}

Theorem \ref{Teo:0Hopfbif} is proven in Section \ref{Sec:TorusBif0Hopf}. Its proof is mainly based on the averaging theory. The role of the averaging theory in the study of the bifurcation of invariant normally hyperbolic tori is discussed in Section \ref{Sec:Preliminary}.

\section{Torus bifurcation via second order analysis}\label{Sec:Preliminary}

In this section, we briefly discuss some recent results in the literature related to the bifurcation of invariant tori. We consider systems of non-autonomous $T$-periodic differential equations, given in the standard form
\begin{equation}\label{eq:Averaging}
\dot{\mathbf{x}}=\varepsilon F_1(t,\mathbf{x},\mu)+\varepsilon^2 F_2(t,\mathbf{x},\mu)+\varepsilon^{3}\tilde{F}(t,\mathbf{x},\mu,\varepsilon),\quad (t,\mathbf{x},\mu,\varepsilon)\in\mathbb{R}\times D\times J\times (-\varepsilon_0,\varepsilon_0),
\end{equation}
where $D$ is an open bounded subset of $\mathbb{R}^2$, $J$ is an open interval and $\varepsilon_0>0$,  and the functions $F_2,F_2,\tilde F$ are of class $C^r$ function, $r>1$ and $T$-periodic in the variable $t$. 

From the periodicity, we can consider \eqref{eq:Averaging} as the following family of autonomous differential systems in the extended phase space $\mathbb{S}^1\times D$, where $\mathbb{S}^1\equiv\mathbb{R}/T\mathbb{Z}$,
\begin{equation}\label{eq:AveragingExtended}
\begin{aligned}
\dot{t}=&1\\
\dot{\mathbf{x}}=&\varepsilon F_1(t,\mathbf{x},\mu)+\varepsilon^2 F_2(t,\mathbf{x},\mu)+\varepsilon^{3}\tilde{F}(t,\mathbf{x},\mu,\varepsilon).
\end{aligned}
\end{equation}
The Poincaré map $\Pi(\mathbf{x},\mu,\varepsilon)$ associated to the equation \eqref{eq:AveragingExtended} defined on the section $\{t=0\}$ is given by
\begin{equation}\label{eq:Poincare}
\Pi(\mathbf{x},\mu,\varepsilon)=\mathbf{x}+\varepsilon \mathbf{f}_{1}(\mathbf{x},\mu)+\varepsilon^2\mathbf{f}_{2}(\mathbf{x},\mu)+O(\varepsilon^{3}),
\end{equation}
where the functions $\mathbf{f}_1$ and $\mathbf{f}_2$ are the \emph{Melnikov functions}. Namely, 
\begin{equation}\label{eq:averageformulas}
\begin{aligned}
\mathbf{f}_1(\mathbf{z})=&\int_{0}^{T}F_1(t,\mathbf{z})dt,\\
\mathbf{f}_2(\mathbf{z})
=&\int_{0}^{T}\left(F_2(t,\mathbf{z})+\partial_xF_1(t,\mathbf{z})\int_{0}^{t}F_1(s,\mathbf{z})ds\right)dt.
\end{aligned}
\end{equation}
We refer the reader to \cite{DouglasEstrobo} for a detailed study of these functions and their relationship with the \emph{averaged functions}.

The use of Melnikov and averaged functions to detect invariant tori in differential systems has been successfully implemented in several papers (see, for instance \cite{CanNov20JDE,Novaes2023,PereiraNovaes23mech}). The main result in this investigation consists in determining generic conditions  for the existence of a curve $\mu(\varepsilon)$ on the parameter space $(\mu,\varepsilon)$ for which the Poincaré map  \eqref{eq:Poincare} undergoes a \emph{Neimark--Sacker bifurcation} \cite{neimark1959some}, which implies the birth of an limit torus from a periodic solution of \eqref{eq:AveragingExtended}. We briefly discuss this method in the next paragraphs.

Suppose that $\mathbf{f}_1(\mathbf{x},\mu)$ is non-vanishing. We need to assume three hypotheses, namely \textbf{H}, \textbf{T} and \textbf{ND}.

\noindent\textbf{H}. \emph{Hopf point hypothesis.} There is a continuous curve $\mu\in J\mapsto \mathbf{x}_\mu\in D$ defined in an interval $J\ni\mu_0$ such that $\mathbf{f}_1(\mathbf{x}_\mu,\mu)\equiv 0$ and the pair of conjugated eigenvalues $\eta(\mu)\pm i\zeta(\mu)$ of $D_{\mathbf{x}}\mathbf{f}_1(\mathbf{x}_\mu;\mu)$ satisfies $\eta(\mu_0)=0$ and $\zeta(\mu_0)=\omega_0>0$.
\medskip

\noindent\textbf{T}. \emph{Transversality.} $\alpha_d=\dfrac{d\eta}{d\mu}(\mu_0)\neq 0$. 

From \textbf{H} (see \cite[Lemma 3]{CanNov20JDE}), we get the existence of a neighborhood $J_0\subset J$ of $\mu_0$, a parameter $\varepsilon_1$, $0<\varepsilon_1<\varepsilon_0$ and a unique function $\boldsymbol{\xi}:J_0\times(-\varepsilon_1,\varepsilon_1)\to\mathbb{R}^2$ such that $$\boldsymbol{\xi}(\mu,0)=\mathbf{x}_\mu\,\,\text{
and } \Pi(\boldsymbol{\xi}(\mu,\varepsilon),\mu,\varepsilon)=\boldsymbol{\xi}(\mu,\varepsilon),\,\,\text{ for all } (\mu,\varepsilon)\in J_0\times (-\varepsilon_1,\varepsilon_1).$$ This implies that equation \eqref{eq:AveragingExtended} admits a unique $T$-periodic orbit $\Phi(t;\mu,\varepsilon)$ satisfying $\Phi(0,\mu,\varepsilon)\to\mathbf{x}_\mu$ as $\varepsilon\to 0$. Now, for each $(\mu,\varepsilon)\in J_0\times(-\varepsilon_1,\varepsilon_1)$, let $\lambda(\mu,\varepsilon)$ and $\overline{\lambda(\mu,\varepsilon)}$ be the pair of complex eigenvalues of $D_{\mathbf{z}}\Pi(\boldsymbol{\xi}(\mu,\varepsilon),\mu,\varepsilon)$. \textbf{T} implies that there exist $\varepsilon_2$, $0<\varepsilon_2<\varepsilon_1$ and a unique smooth function $\mu:(-\varepsilon_2,\varepsilon_2)\to J_0$, with $\mu(0)=\mu_0$ satisfying
\begin{equation*}
|\lambda(\mu(\varepsilon),\varepsilon)|=1,\,\, \lambda(\mu(\varepsilon),\varepsilon)^k\neq 1,\,\text{for }k=1,2,3,4,\,\,\text{and }\left.\dfrac{d}{d\mu}|\lambda(\lambda(\mu(\varepsilon),\varepsilon)|\right\vert_{\mu=\mu(\varepsilon)}\neq 0.
\end{equation*}

The third hypothesis deals with the non-degeneracy of the Lyapunov coefficient $\ell_1^{\varepsilon}$ associated to the Poincaré map $\Pi(\mathbf{x},\mu,\varepsilon)$ at $(\boldsymbol{\xi}(\mu(\varepsilon),\varepsilon))$. The $2$-jet of $\ell_1^{\varepsilon}$ with respect to $\varepsilon$ writes as
\begin{equation*}\label{eq:LyapunovNeimark}
\ell_1^{\varepsilon}=\varepsilon \ell_{1,1}+\varepsilon^{2}\ell_{1,2}+O(\varepsilon^{3}).
\end{equation*}
\noindent\textbf{ND}. \emph{Non-degeneracy.} $(\ell_{1,1})^2+(\ell_{1,2})^2\neq 0$.

The following result is the second-order case of the more general higher-order version established in \cite[Theorem B]{CanNov20JDE}. The normal hyperbolicity of the bifurcated invariant tori was observed in \cite[Theorem 4.1]{Novaes2025} and follows from \cite{chaperon2011generalised}.

\begin{theorem}[{\cite[Theorem B]{CanNov20JDE} and \cite[Theorem 4.1]{Novaes2025}}]\label{Teo:InvariantToriNeimark}
Suppose that $\mathbf{f}_1(\mathbf{x},\mu)$ is non-vanishing. Assume in addition that \textbf{H}, \textbf{T}, and \textbf{ND} hold and let $j^\ast\in\{1,2\}$  be the first subindex such that $\ell_{1,j^\ast}\neq 0$. Then, there exists a curve $\mu(\varepsilon)$, defined in a small neighborhood of $\mu_0$ and satisfying $\mu(0)=\mu_0$, such that for each $\varepsilon$ sufficiently small there exists an interval $I_{\varepsilon}\subset J$ containing $\mu(\varepsilon)$ and an open set $\mathcal{U}_\varepsilon\subset\mathbb{S}^1\times D$ such that
\begin{itemize}
    \item[1.] If $\mu\in I_\varepsilon$ and $\alpha_d\cdot\ell_{1,j^\ast}\cdot(\mu-\mu(\varepsilon))\geqslant0$, equation \eqref{eq:AveragingExtended} has one $T$-periodic orbit  $\Phi(t;\mu(\varepsilon),\varepsilon)\in\mathcal{U}_\varepsilon$ which is unstable (resp. asymptotically stable), provided that $\ell_{1,j^\ast}>0$ (resp. $\ell_{1,j^\ast}<0$). Equation \eqref{eq:AveragingExtended} admits no invariant tori in $\mathcal{U}_\varepsilon$;
    \item[2.] If $\mu\in I_\varepsilon$ and $\alpha_d\cdot\ell_{1,j^\ast}\cdot(\mu-\mu(\varepsilon))<0$, equation \eqref{eq:AveragingExtended} admits a unique invariant torus $T_{\mu,\varepsilon}$ in $\mathcal{U}_\varepsilon$ surrounding the periodic orbit $\Phi(t;\mu,\varepsilon)$. The invariant torus $T_{\mu,\varepsilon}$ is normally hyperbolic and attracting if $\ell_{1,j^\ast}>0$ or repelling if $\ell_{1,j^\ast}<0$.
\end{itemize}
\end{theorem}

 One can explicitly compute the Lyapunov coefficient $\ell_1^{\varepsilon}$, when the Poincaré map satisfies some properties. We remark that these are not restrictive properties, since changes of variables and parameters put the Poincaré map into the convenient form. 
 
 More precisely, applying the change of variables and parameters $\mathbf{x}=\mathbf{y}+\boldsymbol{\xi}(\mu,\varepsilon)$ and $\mu=\sigma+\mu(\varepsilon)$ to the Poincaré map \eqref{eq:Poincare}, yields
\begin{equation*}\label{eq:Poincare_H}
\begin{array}{rl}
H_\varepsilon(\mathbf{y},\sigma)&=\left(H_\varepsilon^1(\mathbf{y},\sigma),H_\varepsilon^2(\mathbf{y},\sigma)\right)=\\
&=\Pi(\mathbf{y}+\boldsymbol{\xi}(\sigma+\mu(\varepsilon),\varepsilon);\sigma+\mu(\varepsilon),\varepsilon)-\boldsymbol{\xi}(\sigma+\mu(\varepsilon),\varepsilon).
\end{array}
\end{equation*}
The expansion of $D_{\mathbf{y}}H_\varepsilon(0,0)$ about $\varepsilon=0$ can be written as
\begin{equation*}
D_{\mathbf{y}}H_\varepsilon(0,0)=Id+\varepsilon A_\varepsilon+O(\varepsilon^{N+1}).
\end{equation*}
Via a linear change of variables $\mathbf{y}=M\cdot\tilde{\mathbf{y}}$, if necessary, we assume that $Id+\varepsilon A_\varepsilon$ is in its Jordan canonical form, more precisely,
$$Id+\varepsilon A_\varepsilon=\left(\begin{array}{cc}
   1+\tilde{\eta}(\varepsilon)  &  -\tilde{\zeta}(\varepsilon)\\
    \tilde{\zeta}(\varepsilon) & 1+\tilde{\eta}(\varepsilon)
\end{array}\right),$$
where $\tilde{\eta}(\varepsilon)=\varepsilon\eta_1+\varepsilon^2\eta_2$ and $\tilde{\zeta}(\varepsilon)=\varepsilon \zeta_1+\varepsilon^2 \zeta_2$, with $\eta_j,\zeta_j\in\mathbb{R}$ for $j\in\{1,2\}$. Thus, we consider the Taylor expansion of $H_\varepsilon(\mathbf{y},0)$ around $(0,0)$:
$$H_{\varepsilon}(\mathbf{y},0)=A \mathbf{x}+\frac{1}{2}B(\mathbf{x,x})+\frac{1}{6}C(\mathbf{x,x,x})+O(\Vert\mathbf{x}\Vert^4),$$
and also the inner product in $\mathbb{C}^2$ given by $\langle \mathbf{u},\mathbf{v}\rangle=\bar{\mathbf{u}}^\intercal\cdot\mathbf{v}$. Let $e^{i\theta_\varepsilon}=\lambda(\mu(\varepsilon),\varepsilon)$ and $\mathbf{p}=(1,-i)/\sqrt{2}$. Then, the Lyapunov coefficient is given by the formula
\begin{eqnarray}\label{eq:Lyapunov}
\ell_1^{\varepsilon}&=&\mathrm{Re}\left(\frac{e^{-i\theta_\varepsilon}\langle\mathbf{p},C(\mathbf{p,p,\bar{p}})\rangle}{2}\right)-\mathrm{Re}\left(\frac{(1-2e^{i\theta_\varepsilon})e^{-2i\theta_\varepsilon}}{2(1-e^{i\theta_\varepsilon})}\langle\mathbf{p},B(\mathbf{p,p})\rangle\langle\mathbf{p},B(\mathbf{p,\bar{p}})\rangle\right)\nonumber\\
&&-\frac{\left\Vert \langle\mathbf{p},B(\mathbf{p,\bar{p}})\rangle\right\Vert^2}{2}-\frac{\left\Vert \langle\mathbf{p},B(\mathbf{\bar{p},\bar{p}})\rangle\right\Vert^2}{4}.
\end{eqnarray}

\section{Torus bifurcation from Hopf-Zero singularities}\label{Sec:TorusBif0Hopf}

This section is devoted to the proof of Theorem \ref{Teo:0Hopfbif}. Before presenting it, we prove Theorem \ref{Teo:0Hopfbifsimples}, a simplified version of Theorem \ref{Teo:0Hopfbif} that illustrates one of the simplest perturbations that can be applied to a differential system with a Hopf-Zero equilibrium satisfying the nondegeneracy condition \eqref{eq:gencond}, leading to the emergence of a normally hyperbolic limit torus. Notably, when the unperturbed system \eqref{unp} is polynomial, this perturbation preserves its degree, an essential feature that will allow us to establish the strict increase of the number $N_h(m)$. The proofs of Theorems \ref{Teo:0Hopfbifsimples} and \ref{Teo:0Hopfbif} are essentially the same, differing only in the complexity of the expressions involved. Therefore, the following proof serves as a didactic version of the proof of Theorem \ref{Teo:0Hopfbif}, where the more cumbersome expressions will be omitted.

\begin{theorem}\label{Teo:0Hopfbifsimples}
Let $P,Q,R$ be $C^3$ functions with no constant or linear terms for which \eqref{eq:gencond} holds. Denote 
$$\beta=-\mathrm{sign}\Big(R^{(2,0,0)}+R^{(0,2,0)} \Big)$$
and
\begin{equation}\label{eq:Nh_Lyapunov}
\parbox{0.9\textwidth}{\raggedright\hangafter=1\hangindent=2em$\displaystyle
\ell_1=-\Omega  \Big(R^{(0,2,0)}+R^{(2,0,0)} \Big)
    \Big(R^{(0,0,2)}  \Big( \Big(R^{(0,2,0)}-R^{(2,0,0)} \Big) \Big(P^{(0,1,1)}+Q^{(1,0,1)} \Big)  
+2 \Big(-Q^{(2,0,0)} \Big(P^{(2,0,0)}+Q^{(1,1,0)}+2
   R^{(1,0,1)} \Big)+2 P^{(1,0,1)}R^{(1,1,0)}+P^{(1,2,0)}
+P^{(1,1,0)}
   P^{(2,0,0)}  +P^{(3,0,0)}-Q^{(0,2,0)}
    \Big(Q^{(1,1,0)}+2
   R^{(1,0,1)} \Big)+Q^{(0,3,0)}+Q^{(2,1,0)}
 +2 R^{(0,2,1)}+2 R^{(2,0,1)} \Big)+2 P^{(0,2,0)}
    \Big(P^{(1,1,0)}+Q^{(0,2,0)} \Big)+4
    \Big(P^{(0,2,0)}+P^{(2,0,0)} \Big)
 R^{(0,1,1)} \Big) -4  \Big(R^{(0,2,0)}+R^{(2,0,0)} \Big)  \Big(3
   P^{(0,0,2)} R^{(0,1,1)}-3 Q^{(0,0,2)}
   R^{(1,0,1)}
+R^{(0,0,3)} \Big)-2 R^{(1,1,0)}
   (R^{(0,0,2)})^2 \Big)+R^{(0,0,2)}
    \Big(R^{(0,2,0)}+R^{(2,0,0)} \Big)^2
    \Big(4  \Big(R^{(0,2,0)}
+R^{(2,0,0)} \Big)
    \Big(P^{(0,0,2)} 
   \Big(2R^{(0,1,1)}-P^{(1,1,0)}-Q^{(0,2,0)} \Big)  +Q^{(0,0,2)}
    \Big(P^{(2,0,0)}
+Q^{(1,1,0)}-2
   R^{(1,0,1)} \Big)-P^{(1,0,2)}-Q^{(0,1,2)} \Big)+R^{(0,0,2)}
    \Big( \Big(R^{(2,0,0)}-R^{(0,2,0)} \Big)
     \Big(P^{(0,1,1)}+Q^{(1,0,1)} \Big)+2
    \Big(-Q^{(2,0,0)}
    \Big(P^{(2,0,0)}+Q^{(1,1,0)} \Big)  -2 P^{(1,0,1)} R^{(1,1,0)}
+P^{(1,2,0)}+P^{(1,1,0)}
   P^{(2,0,0)}+P^{(3,0,0)}+Q^{(0,3,0)}-Q
   ^{(0,2,0)}
   Q^{(1,1,0)}   +Q^{(2,1,0)} \Big)+2 P^{(0,2,0)}
   \Big(P^{(1,1,0)}+Q^{(0,2,0)} \Big) \Big) \Big)+2 \Omega^2 R^{(0,0,2)}
   R^{(1,1,0)}.$}
\end{equation}
Consider the following two-parameter family of differential systems,
\begin{equation}\label{eq:0Hopfperturbado}
\begin{aligned}
\dot{x}=&-y+P(x,y,z),\\
\dot{y}=&x+Q(x,y,z),\\
\dot{z}=&R(x,y,z)+\varepsilon\mu z +\beta\varepsilon^2.
\end{aligned}
\end{equation}
Assume that $\ell_1\neq0$. Then, there exists a curve $\mu(\varepsilon)$, defined in a small neighborhood of $0$, with $\mu(0)=0$, such that for each sufficiently small $\varepsilon$, there exists an interval $I_{\varepsilon}\subset \Lambda$ containing $\mu(\varepsilon)$ with the following property: for every $\mu \in I_{\varepsilon}$ such that $(\mu - \mu(\varepsilon))\ell_1 < 0$, the differential system \eqref{eq:0Hopfperturbado} has a normally hyperbolic limit torus surrounding a periodic solution. Moreover, this limit torus bifurcates from the periodic solution as $\mu$ crosses $\mu(\varepsilon)$, while the periodic solution itself bifurcates from the origin as $\varepsilon$ crosses $0$.
\end{theorem}
\begin{proof}
After applying the rescaling $(x,y,z)\mapsto \varepsilon(x,y,z)$ to system \eqref{eq:0Hopfperturbado}, we obtain the rescaled system
\begin{equation}\label{eq:Nh0Hopfrescaled}
\begin{aligned}
\dot{x} &= -y + \varepsilon P^{(2)} + \varepsilon^2 P^{(3)} + O(\varepsilon^3),\\
\dot{y} &= x + \varepsilon Q^{(2)} + \varepsilon^2 Q^{(3)} + O(\varepsilon^3),\\
\dot{z} &= \varepsilon(R^{(2)} + \mu z + \beta) + \varepsilon^2 R^{(3)} + O(\varepsilon^3),
\end{aligned}
\end{equation}
where, for a $C^r$ function $F:\mathbb{R}^3\to\mathbb{R}$, we denote by $F^{(m)}$ the homogeneous term of degree $m$ in the Taylor expansion of $F$ about the origin.

Next, we apply the cylindrical change of variables, $(x,y,z) =( r \cos\theta, r \sin\theta,w)$, to system \eqref{eq:Nh0Hopfrescaled}. Since $\dot{\theta} = 1 + O(\varepsilon)$, we perform a time rescaling to take $\theta$ as the new independent variable, obtaining the system
\begin{equation*}
\begin{aligned}
r' = \frac{dr}{d\theta} &= \varepsilon F_1^1(\theta,r,w) + \varepsilon^2 F_2^1(\theta,r,w) + O(\varepsilon^3),\\
w' = \frac{dw}{d\theta} &= \varepsilon F_1^2(\theta,r,w) + \varepsilon^2 F_2^2(\theta,r,w) + O(\varepsilon^3),
\end{aligned}
\end{equation*}
which is in the standard form \eqref{eq:Averaging}. From here, we can use the formulae \eqref{eq:averageformulas} to compute the first and second order Melnikov functions $\mathbf{f}_1(r,w)$ and $\mathbf{f}_2(r,w)$.

By denoting $\mathbf{f}_1(r,w) = (\mathbf{f}_1^1, \mathbf{f}_1^2)$, we have
\begin{equation*}\label{eq:Nh_av1}
\begin{aligned}
\mathbf{f}_1^1(r,w) &= \pi \left(P^{(1,0,1)} + Q^{(0,1,1)}\right) r w,\\
\mathbf{f}_1^2(r,w) &= \frac{1}{2} \pi \left(4\beta + 4\mu w + (R^{(0,2,0)} + R^{(2,0,0)}) r^2 + 2 R^{(0,0,2)} w^2\right).
\end{aligned}
\end{equation*}
The explicit expression for $\mathbf{f}_2(r,w)$ is omitted for brevity

Let $\Gamma = \sqrt{|R^{(0,2,0)} + R^{(2,0,0)}|}$ and define $(r_0, w_0) = (2/\Gamma, 0)$. Then,
\begin{equation}\label{eq:Nh_hypothesisH}
\mathbf{f}_1(r_0, w_0) = 0,\quad D\mathbf{f}_1(r_0, w_0) = \begin{pmatrix}
0 & \displaystyle \frac{2\pi(Q^{(0,1,1)} + P^{(1,0,1)})}{\Gamma} \\
-2\pi\beta\Gamma & 2\pi\mu
\end{pmatrix}.
\end{equation}
Taking \eqref{eq:gencond} into account and, in addition, that $R^{(2,0,0)} + R^{(0,2,0)} = -\beta \Gamma^2$, we obtain 
\[
P^{(1,0,1)} + Q^{(0,1,1)}=\dfrac{\Omega}{\beta \Gamma^2}.
\]
Thus, the eigenvalues of $D\mathbf{f}_1(r_0,w_0)$ are given by
\begin{equation}\label{eq:Nh_hypothesisH2}
\lambda^\pm(\mu) = \pi\mu \pm \pi\frac{\sqrt{\mu^2 \Gamma^2 - 4\Omega}}{\Gamma},
\end{equation}
and, in particular, for $\mu = 0$, we obtain $\lambda^\pm(0) = \pm i 2\pi\sqrt{\Omega}/\Gamma$. Therefore, the imaginary part of $\lambda^\pm(\mu)$ is nonzero for $\mu \in J := (-2\sqrt{\Omega}/\Gamma,2\sqrt{\Omega}/\Gamma)$.

The Poincaré map associated with system \eqref{eq:Nh0Hopfrescaled} is then given by
\begin{equation}\label{eq:Nh_Poincare}
\Pi(r, w; \mu, \varepsilon) = (r, w) + \varepsilon \mathbf{f}_1(r,w;\mu) + \varepsilon^2 \mathbf{f}_2(r,w;\mu) + O(\varepsilon^3).
\end{equation}

We now verify conditions \textbf{H}, \textbf{T}, and \textbf{ND}. Condition \textbf{H} is satisfied by equations \eqref{eq:Nh_hypothesisH} and \eqref{eq:Nh_hypothesisH2}. In the notation of Section \ref{Sec:Preliminary}, we have:
\begin{equation}\label{eq:Nh_hypothesisT}
\mu_0 = 0,\quad \mathbf{x}_\mu = (r_0, w_0),\quad \eta(\mu) = \pi\mu,\quad i\zeta(\mu) = \pi\frac{\sqrt{\mu^2\Gamma^2 - 4\Omega}}{\Gamma},\quad \omega_0 = \frac{2\pi\sqrt{\Omega}}{\Gamma}.
\end{equation}
Therefore, condition \textbf{T} follows directly from \eqref{eq:Nh_hypothesisT}.

From conditions \textbf{H} and \textbf{T},  the Implicit Function Theorem provides functions
\[
\boldsymbol{\xi}(\mu,\varepsilon) = (r_0, w_0) + \varepsilon(\xi_1(\mu), \xi_2(\mu)) + O(\varepsilon^2),\quad \mu(\varepsilon) = \varepsilon\mu_1 + O(\varepsilon^2),
\]
satisfying  $\Pi(\boldsymbol{\xi}(\mu,\varepsilon), \mu, \varepsilon) = \boldsymbol{\xi}(\mu,\varepsilon)$ for sufficiently small $(\mu, \varepsilon)$. The explicit expressions for $\xi_1(\mu)$, $\xi_2(\mu)$, and $\mu_1$ are given by:
\begin{eqnarray*}
\xi_1(\mu)&=&\frac{1}{12 \beta  \Gamma ^3
   \Omega}\left(3 \beta ^2 \Gamma ^4 \mu 
   \left(P^{(0,1,1)}+Q^{(1,0,1)}\right)-2
   \beta  \Gamma  \left(4 \Omega \left(2
   \left(P^{(1,1,0)}+Q^{(0,2,0)}\right)+Q^{(2,
   0,0)}\right)\right.\right.\\
   &&+3 \Gamma  \mu 
   \left(-R^{(0,2,0)}
   \left(P^{(0,1,1)}+Q^{(1,0,1)}\right)+P^{(0,
   2,0)}
   \left(P^{(1,1,0)}+Q^{(0,2,0)}\right)\right.\\
   &&-Q^{(2,
   0,0)}
   \left(P^{(2,0,0)}+Q^{(1,1,0)}\right)-2
   P^{(1,0,1)}
   R^{(1,1,0)}+P^{(1,2,0)}+P^{(1,1,0)}
   P^{(2,0,0)}+P^{(3,0,0)}\\
   &&\left.\left.\left.+Q^{(0,3,0)}-
   Q^{(0,2,0)}
   Q^{(1,1,0)}+Q^{(2,1,0)}\right)\right)-6
   \Omega \mu  R^{(1,1,0)}\right),\\
\xi_2(\mu)&=&\frac{1}{4\Gamma^2\Omega} \left(\Gamma^2\beta 
   \left(\left(P^{(0,1,1)}+Q^{(1,0,1)}\right)
   \left(2 R^{(0,2,0)}+\beta  \Gamma ^2\right)-2   \left(P^{(0,2,0)}\left(P^{(1,1,0)}+Q^{(0,2,0)}\right)\right.\right.\right.\\
   &&-Q^{(2,
   0,0)}
   \left(P^{(2,0,0)}+Q^{(1,1,0)}\right)-2
   P^{(1,0,1)}
   R^{(1,1,0)}+P^{(1,2,0)}+P^{(1,1,0)}
   P^{(2,0,0)}+P^{(3,0,0)}\\
   &&\left.\left.\left.+Q^{(0,3,0)}-
   Q^{(0,2,0)}
   Q^{(1,1,0)}+Q^{(2,1,0)}\right)\right)-6\Omega R^{(1,1,0)}\right),\\
\mu_1&=&-\frac{1}{4 \beta  \Gamma ^4\Omega}\left(\beta ^3 \Gamma ^6 R^{(0,0,2)}
   \left(P^{(0,1,1)}+Q^{(1,0,1)}\right)+\beta
   ^2 \Gamma ^4 \left(\Omega
   \left(P^{(0,1,1)}+Q^{(1,0,1)}\right)\right.\right.\\
   &&-2 R^{(0,0,2)} \left(-R^{(0,2,0)}
   \left(P^{(0,1,1)}+Q^{(1,0,1)}\right)+P^{(0,
   2,0)}
   \left(P^{(1,1,0)}+Q^{(0,2,0)}\right)\right.\\
   &&-Q^{(2,0,0)}
   \left(P^{(2,0,0)}+Q^{(1,1,0)}\right)-2
   P^{(1,0,1)}
   R^{(1,1,0)}+P^{(1,2,0)}+P^{(1,1,0)}
   P^{(2,0,0)}+P^{(3,0,0)}\\
   &&\left.\left.+Q^{(0,3,0)}-
   Q^{(0,2,0)}
   Q^{(1,1,0)}+Q^{(2,1,0)}\right)\right)+2
   \beta  \Gamma ^2 \Omega \left(R^{(0,2,0)}
   \left(P^{(0,1,1)}+Q^{(1,0,1)}\right)\right.\\
   &&-Q^{(2,0,0)}
   \left(P^{(2,0,0)}+Q^{(1,1,0)}+2
   R^{(1,0,1)}\right)+P^{(0,2,0)}
   \left(P^{(1,1,0)}+Q^{(0,2,0)}\right)\\
   &&-R^{(1,
   1,0)} \left(R^{(0,0,2)}-2
   P^{(1,0,1)}\right)+2
   \left(P^{(0,2,0)}+P^{(2,0,0)}\right)
   R^{(0,1,1)}+P^{(1,2,0)}\\
   &&+P^{(1,1,0)}P^{(2,0,0)}+P^{(3,0,0)}-Q^{(0,2,0)}
   \left(Q^{(1,1,0)}+2
   R^{(1,0,1)}\right)+Q^{(0,3,0)}+Q^{(2,1,0)}\\
   &&\left.\left.+2 R^{(0,2,1)}+2 R^{(2,0,1)}\right)-2
   \Omega^2 R^{(1,1,0)}\right).
\end{eqnarray*}

Next, consider the change of coordinates defined by
\[
(r, w) = M \cdot (u, v) + \boldsymbol{\xi}(\mu, \varepsilon),
\]
with $\mu = \sigma + \mu(\varepsilon)$, where $M \in \mathbb{R}^{2 \times 2}$ is given by
\[
M = M_0 + \varepsilon M_1,
\]
and
\[
M_0 = \begin{pmatrix}
1 & 0 \\
0 & -\dfrac{\beta\Gamma^2}{\sqrt{\Omega}}
\end{pmatrix},\quad
M_1 = \begin{pmatrix}
0 & 0 \\
m_{21}^1 & m_{22}^1
\end{pmatrix}.
\]
The entries $m_{21}^1$ and $m_{22}^1$ are explicitly given by
\[
\begin{aligned}
m_{21}^1 &= -\frac{1}{4 \Gamma \Omega} \Big(
\beta^2 \Gamma^4 P^{(0,1,1)} + 2\beta \Gamma^2 P^{(0,2,0)} P^{(1,1,0)} + 2\beta \Gamma^2 P^{(1,2,0)} + 2\beta \Gamma^2 P^{(1,1,0)} P^{(2,0,0)} \\
&\quad + 2\beta \Gamma^2 P^{(3,0,0)} + 2\beta \Gamma^2 P^{(0,2,0)} Q^{(0,2,0)} - 2\beta \Gamma^2 P^{(2,0,0)} Q^{(2,0,0)} - 2\beta \Gamma^2 P^{(0,1,1)} R^{(0,2,0)} \\
&\quad - 4\beta \Gamma^2 P^{(1,0,1)} R^{(1,1,0)} + \beta^2 \Gamma^4 Q^{(1,0,1)} + 2\beta \Gamma^2 Q^{(0,3,0)} - 2\beta \Gamma^2 Q^{(0,2,0)} Q^{(1,1,0)} \\
&\quad - 2\beta \Gamma^2 Q^{(1,1,0)} Q^{(2,0,0)} + 2\beta \Gamma^2 Q^{(2,1,0)} - 2\beta \Gamma^2 Q^{(1,0,1)} R^{(0,2,0)} + 6\Omega R^{(1,1,0)} \Big), \\
m_{22}^1 &= \frac{2\beta\Gamma}{3\sqrt{\Omega}} \left( -2 P^{(1,1,0)} - 2 Q^{(0,2,0)} - Q^{(2,0,0)} + 3 R^{(0,1,1)} \right).
\end{aligned}
\]
We note that $M$ is nonsingular for sufficiently small $\varepsilon$, since $\det M=-\beta\Gamma^2/\sqrt{\Omega}+O(\varepsilon)$. Under this change of coordinates, the Poincaré map \eqref{eq:Nh_Poincare} transforms into
 \begin{equation}\label{eq:Nh_Hmap}
\begin{array}{rl}
H_\varepsilon(u,v,\sigma)=\Pi(u+\boldsymbol{\xi}_1(\sigma+\mu(\varepsilon),\varepsilon),v+\boldsymbol{\xi}_2(\sigma+\mu(\varepsilon),\varepsilon),\sigma+\mu(\varepsilon),\varepsilon)-\boldsymbol{\xi}(\sigma+\mu(\varepsilon),\varepsilon).
\end{array}
\end{equation}

We now compute the expansion of the Jacobian $D_{(u,v)}H_\varepsilon(0,0)$ around $\varepsilon = 0$. It takes the form
\[
D_{(u,v)}H_\varepsilon(0,0) = I + \varepsilon A_1 + \varepsilon^2 A_2 + O(\varepsilon^3),
\]
with
\[
A_1 = \begin{pmatrix}
0 & -\dfrac{2\pi\sqrt{\Omega}}{\Gamma} \\
\dfrac{2\pi\sqrt{\Omega}}{\Gamma} & 0
\end{pmatrix},\quad
A_2 = \begin{pmatrix}
-\dfrac{2\pi^2 \Omega}{\Gamma^2} & 0 \\
0 & -\dfrac{2\pi^2 \Omega}{\Gamma^2}
\end{pmatrix}.
\]
Since $Id+\varepsilon A_\varepsilon$ is in the Jordan normal form, we can apply the formula \eqref{eq:Lyapunov} to compute the first Lyapunov coefficient $\ell_1^\varepsilon$ associated to the map \eqref{eq:Nh_Hmap}, obtaining
\[
\ell_1^\varepsilon = \frac{\pi \varepsilon^2}{16 \Omega^2} \ell_1 + O(\varepsilon^3).
\]
Finally, since condition \textbf{ND} holds, the result follows from Theorem \ref{Teo:InvariantToriNeimark}
\end{proof}

\subsection{Proof of Theorem \ref{Teo:0Hopfbif}}

The proof of Theorem~\ref{Teo:0Hopfbif} follows the same strategy as the previous result. However, in this more general setting, the expressions involved become substantially more intricate, and are therefore omitted for the sake of brevity.

We note that the functions $U_i$, $V_i$, and $W_i$ (for $i = 1, 2, 3$) depend on the parameter $\mu$. To simplify the notation, this dependence will be suppressed throughout the proof.

We begin by applying the rescaling $(x,y,z) \mapsto \varepsilon(x,y,z)$, which transforms the system into
\begin{equation}\label{eq:appNh0Hopfrescaled}
\begin{array}{l}
\dot{x}=-y+\varepsilon\left(P^{(2)}+U_1^{(1)}+U_2^{(0)}\right)+\varepsilon^2\left(P^{(3)}+U_1^{(2)}+U_2^{(1)}+U_3^{(0)}\right)+O(\varepsilon^3),\\
\dot{y}=x+\varepsilon\left(Q^{(2)}+V_1^{(1)}+V_2^{(0)}\right)+\varepsilon^2\left(Q^{(3)}+V_1^{(2)}+V_2^{(1)}+V_3^{(0)}\right)+O(\varepsilon^3),\\
\dot{z}=\varepsilon\left(R^{(2)}+W_1^{(1)}+W_2^{(0)}\right)+\varepsilon^2\left(R^{(3)}+W_1^{(2)}+W_2^{(1)}+W_3^{(0)}\right)+O(\varepsilon^3).
\end{array}
\end{equation}

Next, we introduce cylindrical coordinates via the change of variables $(x,y,z) = (r \cos\theta, r \sin\theta, w)$. Since $\dot{\theta} = 1 + O(\varepsilon)$, we perform a time rescaling and take $\theta$ as the new independent variable, thereby reducing system \eqref{eq:appNh0Hopfrescaled} to the standard form~\eqref{eq:Averaging}.

We then compute the first two Melnikov functions, $\mathbf{f}_1(r,w)$ and $\mathbf{f}_2(r,w)$, using the formulas given in~\eqref{eq:averageformulas}. The expression for $\mathbf{f}_1(r,w) = (\mathbf{f}_1^1,\mathbf{f}_1^2)$ is given by
\begin{equation*}\label{eq:Nh_av1}
\begin{array}{l}
\mathbf{f}_1^1(r,w)=\pi  r \left((P^{(1,0,1)}+
   Q^{(0,1,1)}) w+U_1^{(1,0,0)}+V_1^{(
   0,1,0)}\right),\\
\mathbf{f}_1^2(r,w)=\frac{1}{2} \pi \left((R^{(0,2,0)}+
   R^{(2,0,0)})r^2+2R^{(0,0,2)}w^2+4 W_1^{(0,0,1)}w+4 W_2(0,0,0)\right),
\end{array}
\end{equation*}   
As in the previous proof, the expression for $\mathbf{f}_2(r,w)$ is omitted due to its length.

Now, assuming condition \eqref{eq:appcond1}, define
\[
r_{\mu}=\sqrt{-\Gamma(\mu)}\quad\text{and}\quad w_{\mu}=-\frac{U_1^{(1,0,0)}+V_1^{(0,1,0)}}{P^{(1,0,1)}+Q^{(0,1,1)}}.
\] 
Note that $\mathbf{f}_1(r_{\mu},w_{\mu}) = 0$, and the eigenvalues of the Jacobian matrix $D\mathbf{f}1(r_{\mu},w_{\mu})$ are given by
\begin{equation*}
\lambda^\pm(\mu)=\eta(\mu) \pm\sqrt{\eta(\mu)^2-\pi^2 \Omega r_{\mu}^2}.
\end{equation*}
In particular, for $\mu = \mu_0$, we have $\lambda^\pm(\mu_0) = \pm i\pi \sqrt{\Omega} r_{\mu_0}$. Let $J_{0} \subset \mathbb{R}$ be a small open interval containing $\mu_{0}$ such that $\eta(\mu)^2 - \pi^2 \Omega r_{\mu}^2 < 0$ for all $\mu \in J_0$.

The Poincaré map associated with system~\eqref{eq:appNh0Hopfrescaled} is then of the form
\begin{equation}\label{eq:appNh_Poincare}
\Pi(r, w; \mu, \varepsilon) = (r, w) + \varepsilon \mathbf{f}_1(r,w;\mu) + \varepsilon^2 \mathbf{f}_2(r,w;\mu) + O(\varepsilon^3).
\end{equation}

We now verify conditions \textbf{H}, \textbf{T}, and \textbf{ND}. From the above discussion, it is clear that conditions \textbf{H} and \textbf{T} are satisfied. In the notation of Section \ref{Sec:Preliminary}, we have
\begin{eqnarray*}\label{eq:appNh_Neimark1}
\mathbf{x}_\mu=(r_{\mu},w_{\mu}),\,\,\zeta(\mu)=\sqrt{\pi^2 \Omega r_{\mu}^2-\eta(\mu)^2},\,\, \text{ and }\,\,\omega_{\mu}=\pi\sqrt{\Omega} r_{\mu}.
\end{eqnarray*}
By the Implicit Function Theorem, conditions \textbf{H} and \textbf{T} guarantee the existence of functions 
\[
\boldsymbol{\xi}(\mu,\varepsilon) = (r_\mu,w_\mu)+ \varepsilon(\xi_1(\mu), \xi_2(\mu)) + O(\varepsilon^2)\quad\text{and} \quad \mu(\varepsilon) = \varepsilon\mu_1 + O(\varepsilon^2),
\]
satisfying  $\Pi(\boldsymbol{\xi}(\mu,\varepsilon), \mu, \varepsilon) = \boldsymbol{\xi}(\mu,\varepsilon)$ for sufficiently small $(\mu, \varepsilon)$. 
Although the expressions for $\xi_1(\mu)$, $\xi_2(\mu)$, and $\mu_1$ can be computed explicitly, they are omitted due to their length.

Next, consider the change of coordinates defined by
\[
(r, w) = M \cdot (u, v) + \boldsymbol{\xi}(\mu, \varepsilon),
\]
with $\mu = \sigma + \mu(\varepsilon)$, where $M \in \mathbb{R}^{2 \times 2}$ is given by
\[
M = M_0 + \varepsilon M_1,
\]
and
\[
M_0 = \begin{pmatrix}
1 & 0 \\
0 & -\dfrac{\beta\Gamma^2}{\sqrt{\Omega}}
\end{pmatrix},\quad
M_1 = \begin{pmatrix}
0 & 0 \\
m_{21}^1 & m_{22}^1
\end{pmatrix}.
\]
The entries $m_{21}^1$ and $m_{22}^1$ can be computed explicitly, but they are omitted due to their length.
We note that $M$ is nonsingular for sufficiently small $\varepsilon$, since $\det M=-\beta\Gamma^2/\sqrt{\Omega}+O(\varepsilon)$. 
The Poincaré map \eqref{eq:appNh_Poincare} is then transformed into 
\begin{equation*}\label{eq:appNh_H}
\begin{array}{rl}
H_\varepsilon(u,v,\sigma)=\Pi(u+\boldsymbol{\xi}_1(\sigma+\mu(\varepsilon),\varepsilon),v+\boldsymbol{\xi}_2(\sigma+\mu(\varepsilon),\varepsilon),\sigma+\mu(\varepsilon),\varepsilon)-\boldsymbol{\xi}(\sigma+\mu(\varepsilon),\varepsilon).
\end{array}
\end{equation*}

We now compute the expansion of the Jacobian $D_{(u,v)}H_\varepsilon(0,0)$ around $\varepsilon = 0$. It takes the form
\begin{equation}
D_{(u,v)}H_\varepsilon(0,0)=Id+\varepsilon A_1+\varepsilon^2A_2+O(\varepsilon^3),
\end{equation}
 with
\begin{equation*}
A_1=\left(
\begin{array}{cc}
 0 & -\pi  \sqrt{\Omega}  r_{\mu_0} \\
 \pi  \sqrt{\Omega} r_{\mu_0} & 0 \\
\end{array}
\right),\quad A_2=\left(
\begin{array}{cc}
\displaystyle -\frac{1}{2} \pi ^2 \Omega (r_{\mu_0})^2 & -\zeta_2 \\
\zeta_2 &\displaystyle -\frac{1}{2} \pi ^2 \Omega (r_{\mu_0})^2\\
\end{array}
\right),
\end{equation*}
where $\zeta_2$ is a cubic polynomial on the coefficients of the differential system \eqref{eq:app0Hopfperturbado}. 

Since $Id+\varepsilon A_\varepsilon$ is in the Jordan normal form, we can apply the formula \eqref{eq:Lyapunov} to compute the first Lyapunov coefficient $\ell_1^\varepsilon$ associated to the map \eqref{eq:appNh_H},  obtaining
$$\ell_1^\varepsilon=\frac{\pi\varepsilon^2}{16\Omega^2}\ell_1+O(\varepsilon^3),$$
where $\ell_1$ is given by the expression \eqref{eq:Nh_Lyapunov}. Since \textbf{ND} is satisfied, by Theorem \ref{Teo:InvariantToriNeimark}, the result follows. 

\subsection{Example}
The following two-parameter family of differential systems provides an illustrative example of the bifurcation described in Theorem~\ref{Teo:0Hopfbifsimples}:
\begin{equation}\label{eq:ex1}
\begin{aligned}
\dot{x} &= -y,\nonumber\\
\dot{y} &= x + y z,\\
\dot{z} &= -x^2 + xy + z^2 + \varepsilon \mu z + \varepsilon^2.\nonumber
\end{aligned}
\end{equation}
For this system, we have $\Omega = 2$ and $\ell_1 = -48$. Moreover, we compute
\[
\mu(\varepsilon) = \frac{3}{4}\varepsilon  + \mathcal{O}(\varepsilon^2).
\]
Figure \ref{Fig:ex1} shows numerical simulations of several trajectories of the system above, rescaled via $(x,y,z) \mapsto \varepsilon (x,y,z)$, for $(\mu, \varepsilon) = (0.05, 0.05)$.

\begin{figure}[ht]
\centering
 \begin{overpic}[width=1\linewidth]{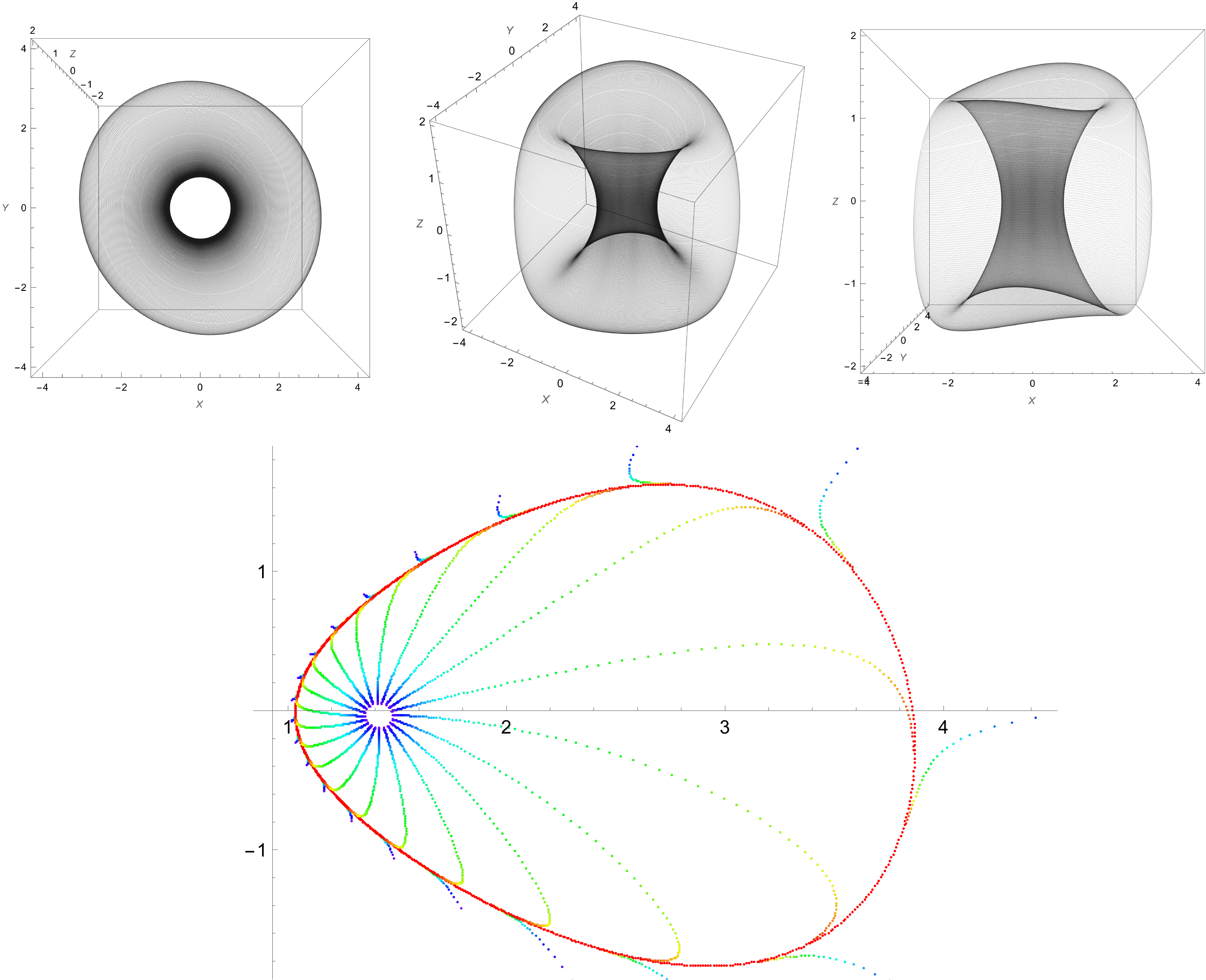}
    \end{overpic}
\bigskip
    \caption{Numerical simulations of several trajectories of system \eqref{eq:ex1} for $(\mu,\varepsilon) = (0.05, 0.05)$. The bottom panel shows the Poincaré map defined on the section $y = 0$, while the top panel shows some views of the phase portrait of \eqref{eq:ex1} rescaled via $(x,y,z) \mapsto \varepsilon(x,y,z)$.}
    \label{Fig:ex1}
\end{figure}

\section{Strict increase of $N_h(m)$: Proof of Theorem \ref{Teo:StrictIncrease}}\label{Sec:StrictIncrease}

This section is dedicated to the proof of Theorem \ref{Teo:StrictIncrease}. We begin with the following lemma, which extends \cite[Lemma 1]{SantanaGasull} to the higher dimensional setting and plays a key role in the proof.

\begin{lemma}\label{Lema:intersecB}
Let $X$ be an $n$-dimensional polynomial vector field of degree $m$, with $n \geq 2$ and $m \geq 1$, and let $B \subset \mathbb{R}^n$ be a closed ball centered at the origin. Then, there exist an arbitrarily small polynomial perturbation $\tilde{X}$ of $X$ of degree $m$ and a hyperplane $\Sigma$ such that $\tilde{X}$ has a regular point $p \in \mathbb{R}^n \setminus B$, the line $\{p + \lambda \tilde{X}(p) :\, \lambda \in \mathbb{R}\}$ is contained in $\Sigma$, and $\Sigma \cap B = \emptyset$.
\end{lemma}
\begin{proof}
Let $X = (X_1, \dots, X_n)$. By an arbitrarily small perturbation of $X$, we may assume that the polynomials
\[
f(x_1,x_2) := X_1(x_1, x_2, 0, \dots, 0) \quad \text{and} \quad g(x_1,x_2) := X_2(x_1, x_2, 0, \dots, 0)
\]
have no common factor. Then, by Bézout’s Theorem, the system $f = g = 0$ has only finitely many solutions in $\mathbb{R}^2$, so $X$ has only finitely many singularities in the $x_1x_2$-plane.

Let $f^{(m)}$ and $g^{(m)}$ denote the homogeneous components of degree $m$ of $f$ and $g$, respectively. After a further arbitrarily small perturbation, we may assume that $f^{(m)}(1, 0) g^{(m)}(1, 0) \neq 0$. Define $p_x := (x, 0, \dots, 0)$ with $x > 0$. Since the singularities of $X$ in the $x_1x_2$-plane are finite, there exists $x_0 > 0$ such that for all $x > x_0$, $p_x$ is a regular point of $\tilde{X}$.

Now consider the projection $\pi : \mathbb{R}^n \to \mathbb{R}^2$ onto the first two coordinates, and let $\pi(B) \subset \mathbb{R}^2$ be the projection of the ball $B$. For each $x > 0$, let $l_x^\pm$ be the two straight lines tangent to $\pi(B)$ passing through the point $(x, 0)$, and let $\theta^\pm(x)$ be the angles between $l_x^\pm$ and the $x_1$-axis.

Define the line
\[
l_x := \{ \pi\big(p_x + \lambda \tilde{X}(p_x)\big) : \lambda \in \mathbb{R} \} = \{ (x, 0) + \lambda \pi(\tilde{X}(p_x)) : \lambda \in \mathbb{R} \},
\]
and let $\varphi(x)$ be the angle between $l_x$ and the $x_1$-axis. As $x \to \infty$, we have
\[
\lim_{x \to \infty} \theta^\pm(x) = 0, \quad \text{and} \quad \lim_{x \to \infty} \varphi(x) = \lim_{x \to \infty} \arctan\left( \frac{g(x, 0)}{f(x, 0)} \right) = \arctan\left( \frac{g^{(m)}(1,0)}{f^{(m)}(1,0)} \right) \neq 0.
\]
Therefore, for $x$ sufficiently large, $|\varphi(x)| > |\theta^\pm(x)|$, which implies that the line $l_x$ does not intersect $\pi(B)$. Let $x^*$ be such a value.
Taking $p= p_{x^*}$ and $\Sigma=\pi^{-1}(l_{x^\ast})$, the result holds.
\end{proof}

\begin{figure}[ht]
\centering
		\begin{overpic}[width=0.7\linewidth]{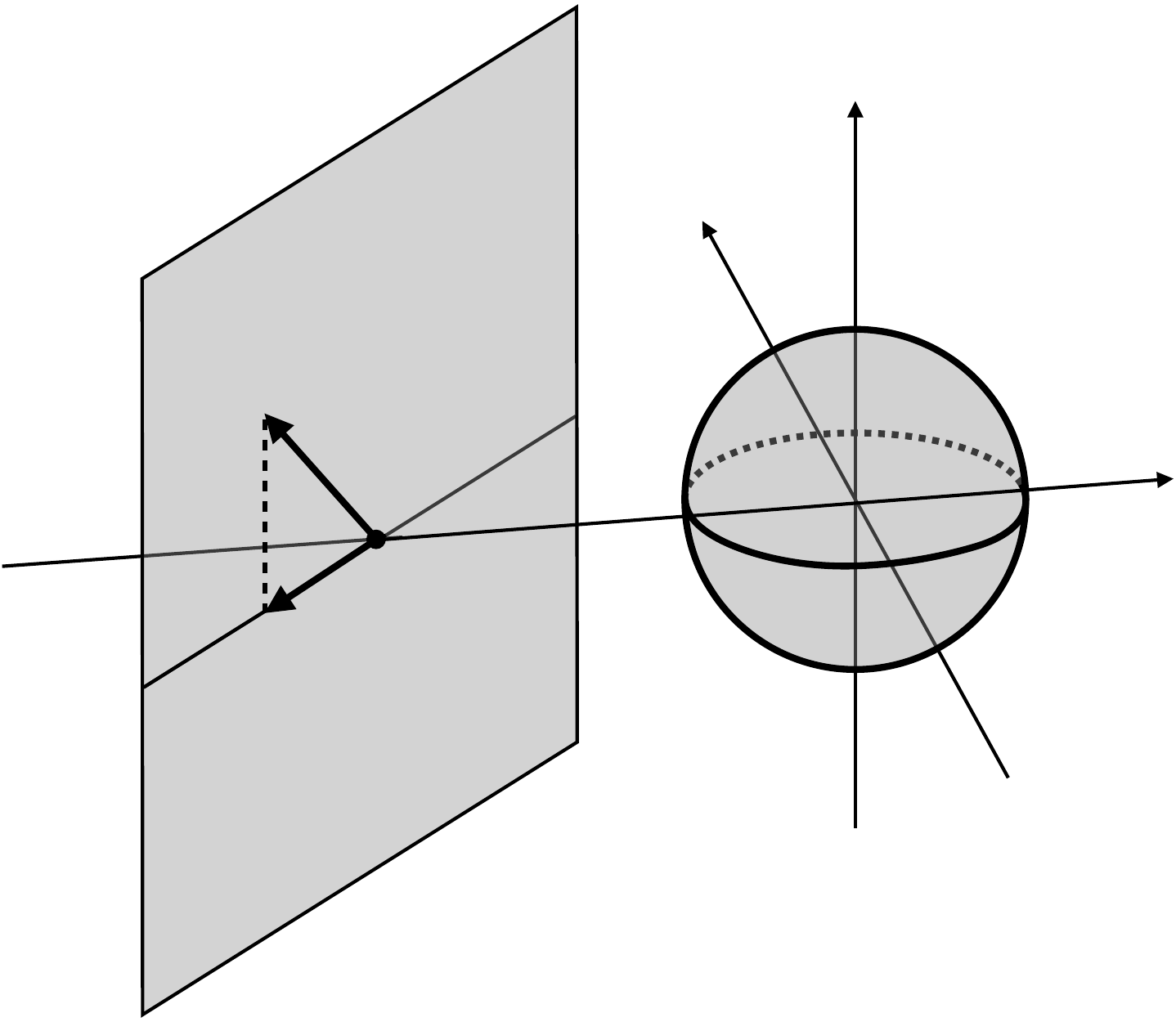}
			\put(32,38){$p$}
			\put(20,52.5){$\tilde{X}(p)$}
			\put(20,30){$\pi(\tilde{X}(p))$}
			\put(61,45.5){$\pi(B)$}
			\put(77,59){$B$}
			\put(13,60){$\Sigma=\pi^{-1}(l_{x^*})$}
			\put(101,46){$x_1$}
			\put(58,70){$x_2$}
			\put(71,80){$\mathbb{R}^{n-2}$}
		\end{overpic}
	\caption{Schematic of the proof of Lemma \ref{Lema:intersecB}.}
	\label{Fig:ex2}
\end{figure}
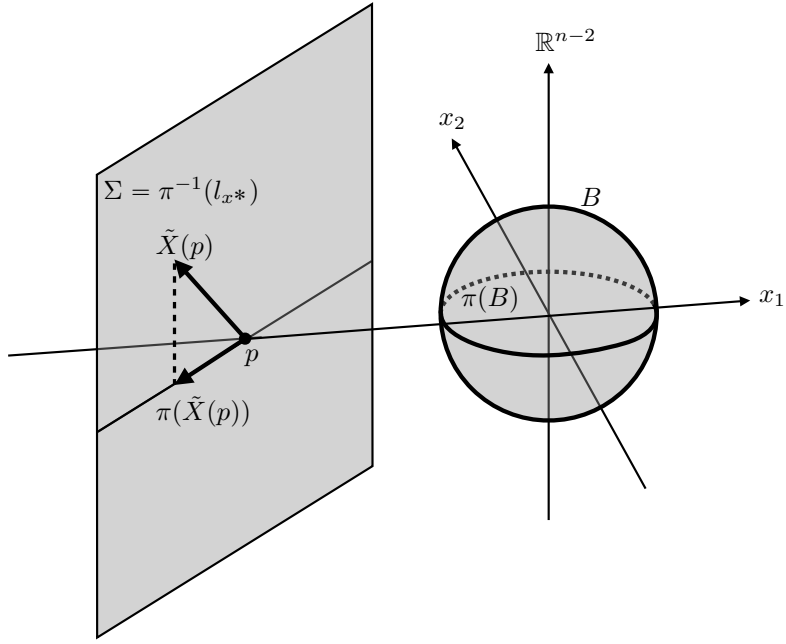

\subsection{Proof of Theorem \ref{Teo:StrictIncrease}}
Let $X$ be a three-dimensional polynomial vector field of degree $m$ with $N_h(m)$ normally hyperbolic limit tori. Let $B \subset \mathbb{R}^3$ be a closed ball centered at the origin that contains all such tori in its interior.  By Lemma \ref{Lema:intersecB} and Fenichel's Theorem~\cite{Fenichel1}, there exists an arbitrarily small perturbation $\tilde{X}$ of $X$, also of degree $m$, such that $\tilde{X}$ has $N_h(m)$ normally hyperbolic limit tori contained in the interior of $B$, and there exists a regular point $p \in \mathbb{R}^3 \setminus B$ such that the vertical plane $\Sigma$ contains the straight line $\{p + \lambda \tilde{X}(p):\,\lambda\in\mathbb{R}\}$ and satisfies $\Sigma \cap B = \emptyset$. By applying a translation, we may assume that $p$ is the origin, particularly, $\tilde{X}(0)\neq0$. In this case, $B$ is no longer centered at the origin, but the crucial fact is that the hyperplane $\Sigma$ still does not intersect $B$ and, consequently, avoids all $N_h(m)$ limit tori.

Let $\tilde{X}$ be the vector field associated to the system
\begin{equation*}
\tilde{X}:\left\{\begin{array}{ll}
		\dot{x} = \tilde{P}(x,y,z),\\
		\dot{y} = \tilde{Q}(x,y,z),\\
		\dot{z} = \tilde{R}(x,y,z),
	\end{array}\right.
\end{equation*}
and let $\Sigma$ be the plane defined by $a_0 x + b_0 y + c_0 z = 0$. Since $\Sigma$ is vertical and contains $\tilde{X}(0)$, it follows that
\[
a_0 = \tilde{Q}(0), \quad b_0 = -\tilde{P}(0), \quad \text{and} \quad c_0 = 0.
\]

Consider the vector field $X_{L,\delta}$ associated to the following two-parameter family of polynomial differential systems of degree $m+1$:
\begin{equation*}\label{eq:Xm+1}
X_{L,\delta}:\left\{\begin{array}{ll}
\dot{x} =\; & \left((a_0 + \delta a_1) x + b_0 y\right)\tilde{P}(x,y,z),\\
\dot{y}=\;& \left((a_0 + \delta a_2(L,\delta)) x + (b_0 + \delta b_2) y\right)\tilde{Q}(x,y,z),\\
\dot{z}=\;& (a_0 x + b_0 y)\tilde{R}(x,y,z),
\end{array}\right.
\end{equation*}
where
\[
a_1 = \frac{L}{\tilde{P}(0)}, \quad b_2 = -\frac{L}{\tilde{Q}(0)}, \quad \text{and} \quad a_2(L,\delta) = \frac{1 + 2\tilde{P}(0)\tilde{Q}(0)L + L^2\delta}{\tilde{P}(0)^2 \tilde{Q}(0)}.
\]
The parameters $\delta$ and $L$ are positive ant will be chosen later.

Note that $X_{L,0} = (a_0 x + b_0 y)\tilde{X}$, which coincides with $\tilde{X}$ on $\mathbb{R}^3 \setminus \Sigma$, up to time reparametrization. In particular, $X_{L,0}$ has $N_h(m)$ normally hyperbolic limit tori contained in the interior of $B$. By Fenichel's Theorem, for each fixed $L \in \mathbb{R}$, there exists $\tilde{\delta}(L) > 0$ such that for all $\delta \in (0, \tilde{\delta}(L))$, the vector field $X_{L,\delta}$ also exhibits $N_h(m)$ normally hyperbolic limit tori inside $B$.

Moreover, the characteristic polynomial of the Jacobian matrix $JX_{L,\delta}(0)$ is
\[
p(\lambda) = -\lambda^3 - \delta\lambda.
\]
Thus, for each $L\in\mathbb{R}$, there exists  $\bar{\delta}(L)\in(0,\tilde{\delta}(L)]$ such that the origin is a Hopf-Zero equilibrium of $X_{L,\delta}$ for every $\delta \in (0, \bar{\delta}(L))$.

Now, to apply Theorem \ref{Teo:0Hopfbifsimples} and produce an additional limit torus bifurcating from $X_{L,\delta}$, consider the time rescaling $\bar{t} = \sqrt{\delta} t$ and the linear change of variables $(x,y,z) = M(\bar{x}, \bar{y}, \bar{z})$, where
\[
M = \begin{pmatrix}
1 & 0 & 0 \\
\frac{L\delta + \tilde{P}(0)\tilde{Q}(0)}{\tilde{P}(0)^2} & \frac{\sqrt{\delta}}{\tilde{P}(0)^2} & 0 \\
\frac{\tilde{R}(0)}{\tilde{P}(0)} & -\frac{L\sqrt{\delta}\tilde{R}(0)}{\tilde{P}(0)} & 1
\end{pmatrix}.
\]
This transformation brings the linear part of $X_{L,\delta}$ into its Jordan normal form, yielding a new vector field $Y_{L,\delta}$ whose linear part is $(-y,x,0)$.

Let $\Omega(L,\delta)$ denote the expression defined in \eqref{eq:gencond} for the vector field $Y_{L,\delta}$. A direct computation shows that
\[
\Omega(L,\delta) = A(L) + \delta\, B(L,\delta),
\]
where $A(L)$ is quadratic in $L$, and $B(L,\delta)$ is quartic in $L$ and linear in $\delta$. Moreover, one can see that
\[
\lim_{L \to \infty} \dfrac{A(L)}{L^2} = \frac{2\tilde{R}(0)^2 \left( \tilde{Q}(0)\, \tilde{P}^{(0,0,1)} - \tilde{P}(0)\, \tilde{Q}^{(0,0,1)} \right)^2}{\tilde{P}(0)^2} > 0.
\]
It follows that there exists $L^* > 0$ sufficiently large such that $A(L^*) > 0$. Consequently, there exists $\delta^* \in (0, \bar{\delta}(L^*))$ for which $\Omega(L^*, \delta^*) > 0$, implying that condition \eqref{eq:gencond} holds for the vector field $Y_{L^*, \delta^*}$.

Let now $\ell_1$ denote the expression defined in \eqref{eq:Nh_Lyapunov} for $Y_{L^*, \delta^*}$. After a small perturbation, if necessary, we may assume that $\ell_1 \neq 0$. Then, by Theorem \ref{Teo:0Hopfbif}, there exists a polynomial vector field $Y$ of degree $m+1$, arbitrarily close to $Y_{L^*, \delta^*}$, having a normally hyperbolic limit torus arbitrarily close to the origin. Since, by Fenichel’s Theorem, $Y$ also preserves the original $N_h(m)$ normally hyperbolic limit tori inside $B$, we conclude that $Y$ has $N_h(m) + 1$ such tori. Hence, $N_h(m+1) \geq N_h(m) + 1$, which concludes this proof.

\section*{Acknowledgments}
L.Q. Arakaki is supported by S\~{a}o Paulo Research Foundation (FAPESP) grant 2024/06926-7.  D.D. Novaes is partially supported by S\~{a}o Paulo Research Foundation (FAPESP) grant 2018/13481-0, and by Conselho Nacional de Desenvolvimento Cient\'{i}fico e Tecnol\'{o}gico (CNPq) grant 309110/2021-1.

\section*{Declarations}
\subsection*{Ethical Approval} Not applicable
\subsection*{Competing interests} To the best of our knowledge, no conflict of interest, financial or other, exists.
\subsection*{Authors' contributions} All persons who meet authorship criteria are listed as authors, and all authors certify that they have participated sufficiently in the work to take public responsibility for the content, including participation in the conceptualization, methodology, formal analysis, investigation, writing-original draft preparation and writing-review \& editing.
\subsection*{Availability of data and materials} Data sharing not applicable to this article as no datasets were generated or analyzed during the current study.

\bibliographystyle{siam}
\bibliography{ReferenciasPD.bib}
\end{document}